\documentclass[12pt,a4paper]{article}
\usepackage{geometry}
\usepackage{hyperref}

\usepackage{times,authblk}
\usepackage[UKenglish]{babel}
\renewcommand{\phi}{\varphi}
\renewcommand{\rho}{\varrho}
\renewcommand{\epsilon}{\varepsilon}
\usepackage[cal=esstix]{mathalfa}  
\usepackage{amsmath, amsfonts, amssymb, amsthm, bm, stmaryrd, enumerate, mathtools}
\usepackage{graphicx}
\usepackage{tikz}
\usepackage{tikz-cd}
\tikzset{commutative diagrams/.cd}

\newtheorem{Def}{Definition}[section]
\newenvironment{definition}{\begin{Def} \rm}{\end{Def}}
\newtheorem{lemma}[Def]{Lemma}
\newtheorem{proposition}[Def]{Proposition}
\newtheorem{corollary}[Def]{Corollary}
\newtheorem{theorem}[Def]{Theorem}
\newtheorem{example}[Def]{Example}
\newtheorem{remark}[Def]{Remark}

\newcommand{\komma}{,\hspace{0.3em}}
\newcommand{\id}{\text{\rm id}}
\renewcommand{\leq}{\leqslant}
\renewcommand{\geq}{\geqslant}

\newcommand{\Naturals}{{\mathbb N}}

\newcommand{\Reals}{{\mathbb R}}
\newcommand{\Complexes}{{\mathbb C}}
\newcommand{\Quaternions}{{\mathbb H}}

\newcommand{\notperp}{\mathbin{\not\perp}}

\DeclareMathOperator{\closure}{cl}
\renewcommand{\c}{^\perp}
\newcommand{\cc}{^{\perp\perp}}
\newcommand{\herm}[2]{\left( #1 , #2 \right)}
\newcommand{\lin}[1]{\langle #1\rangle}
\renewcommand{\P}{\mathbf P}
\newcommand{\scmu}{\mathbin{\raisebox{2pt}{$\scriptscriptstyle\diamond$}}}
\newcommand{\zerokernel}[1]{{#1}^{\circ}}
\newcommand{\adj}{^\ast}
\DeclareMathOperator{\kernel}{ker}
\DeclareMathOperator{\image}{im}

\begin{document}

\title{Adjointable maps between linear orthosets}

\author[1]{Jan Paseka}

\author[2]{Thomas Vetterlein}

\affil[1]{\footnotesize Department of Mathematics and Statistics,
Masaryk University \authorcr
Kotl\'a\v rsk\'a 2, 611\,37 Brno, Czech Republic \authorcr
{\tt paseka@math.muni.cz}}

\affil[2]{\footnotesize Institute for Mathematical Methods in Medicine and Data Based Modeling, \authorcr
Johannes Kepler University Linz \authorcr
Altenberger Stra\ss{}e 69, 4040 Linz, Austria \authorcr
{\tt Thomas.Vetterlein@jku.at}}

\date{\today}

\maketitle

\begin{abstract}\parindent0pt\parskip1ex

\noindent Given an (anisotropic) Hermitian space $H$, the collection $\P(H)$ of at most one-dimensional subspaces of $H$, equipped with the orthogonal relation $\perp$ and the zero linear subspace $\{0\}$, is a linear orthoset and up to orthoisomorphism any linear orthoset of rank $\geq 4$ arises in this way. We investigate in this paper the correspondence of structure-preserving maps between Hermitian spaces on the one hand and between the associated linear orthosets on the other hand. Our particular focus is on adjointable maps.

We show that, under a mild assumption, adjointable maps between linear orthosets are induced by quasilinear maps between Hermitian spaces and if the latter are linear, they are adjointable as well. Specialised versions of this correlation lead to Wigner-type theorems; we see, for instance, that orthoisomorphisms between the orthosets associated with at least $3$-dimensional Hermitian spaces are induced by quasiunitary maps.

In addition, we point out that orthomodular spaces of dimension $\geq 4$ can be characterised as irreducible Fr\' echet orthosets such that the inclusion map of any subspace is adjointable. Together with a transitivity condition, we may in this way describe the infinite-dimensional classical Hilbert spaces.



\mbox{}\vspace{-2ex}

\end{abstract}

\section{Introduction}
\label{sec:Introduction}

In a well-known way, a linear space $V$ can be reduced to a structure of an apparently simpler type: the projective space associated with $V$ is the collection of one-dimensional subspaces of $V$, equipped with the ternary relation expressing that a subspace is contained in the linear span of two others. Provided that the rank is at least $4$, the structure thus obtained allows the reconstruction of $V$. The question arises how maps between linear spaces are related to maps between the corresponding projective spaces. Answers are given by various versions of the Fundamental Theorem of Projective Geometry. A thorough discussion of this topic can be found in Faure and Fr\" olicher's monograph on ``Modern Projective Geometry'' \cite{FaFr}. One version, expressed in the terminology of \cite{FaFr}, goes as follows: any semilinear map gives rise to a morphism between the associated projective spaces \cite[Proposition 6.3.5]{FaFr}; and any non-degenerate morphism between projective spaces is induced by a semilinear map between the corresponding linear spaces \cite[Theorem 10.1.3]{FaFr}. We note that the two involved scalar skew fields are not in general isomorphic.

The present paper deals with the case when the linear spaces under consideration are equipped with an inner product. By a {\it Hermitian space}, we mean a linear space together with a symmetric and anisotropic sesquilinear form. Hermitian spaces could be described as orthogeometries, that is, as projective spaces that are additionally endowed with an orthogonality relation \cite[Chapter 14]{FaFr}. However, the relation of linear dependence is in this case redundant: the linear span of two vectors is determined by the orthogonality relation. To describe Hermitian spaces, we may hence focus exclusively on the concept of orthogonality.

Realising that Hermitian spaces, and in particular Hilbert spaces, can be reduced to their orthogonality relation, David Foulis and his collaborators once coined the notion of an orthogonality space: a structure based on nothing but a single binary relation, assumed to be symmetric and irreflexive. We deal here with a slight variant of this notion. We define {\it orthosets} in a similarly simple fashion but require them to possess an additional element $0$ that is orthogonal to all elements.  Orthosets fulfilling moreover a certain combinatorial condition are called {\it linear}. The following correspondence holds: any Hermitian space gives rise to a linear orthoset and, provided that the rank is at least $4$, any linear orthoset arises from an essentially uniquely determined Hermitian space \cite{PaVe3}.

We investigate in this paper the interrelations between the maps on both sides of this correspondence. That is, we consider maps between Hermitian spaces on the one hand and maps between orthosets on the other one. In contrast to their structural simplicity, it is not straightforward to decide which sort of maps between orthosets are actually relevant. One could consider orthogonality-preserving maps as a natural choice, cf.\ \cite{PaVe1,PaVe2}. However, the requirement to preserve orthogonality does not lead in any sense to the preservation of linear dependence. Instead, we focus on the condition of adjointability. A map $f \colon X \to Y$ between orthosets is said to be {\it adjointable} if there is a map $g \colon Y \to X$ such that, for any $x \in X$ and $y \in Y$, $f(x) \perp y$ is equivalent to $x \perp g(y)$. Adjointability turns out to be closely related to the linearity of maps. For instance, any linear map between finite-dimensional Hermitian spaces induces an adjointable map between the corresponding orthosets. Our issue is to formulate converse statements.

We proceed as follows. In the subsequent Section~\ref{sec:Linear-spaces}, we review the aforementioned situation for linear spaces and projective spaces, preparing the ground for what follows. In accordance with our definition of an orthoset, we will slightly deviate from the common definition; the projective space associated with a linear space will be assumed to consist of all subspaces spanned by a single vector. The additional zero element that occurs in this way does not affect the overall concept and yet simplifies matters considerably.

From Section~\ref{sec:Hermitian-spaces} on, we investigate Hermitian spaces and their corresponding linear orthosets. We show that adjointable maps between orthosets preserve, in a natural sense, linear dependence and we establish that any adjointable map whose image is not contained in a $2$-dimensional subspace is induced by a quasilinear map. Moreover, we show that a linear map between Hermitian spaces is, in the usual sense, adjointable if and only if so is the induced map between the orthosets.

It turns out that also orthogonality-preserving maps can be conveniently dealt with in the present framework. Indeed, a bijective map $f$ between orthosets that possesses $f^{-1}$ as an adjoint is the same as an orthoisomorphism, that is, a bijection preserving $\perp$ in both directions. Section~\ref{sec:Orthometric-correspondences} is devoted to the interrelation between orthoisomorphisms and quasiunitary maps. In this context, the work of Robert Piziak is of essential importance \cite{Piz1}. Piziak proved that if a quasilinear map between at least $2$-dimensional Hermitian spaces preserves orthogonality, then it actually preserves the Hermitian form up to a factor. We show that Piziak's result holds under weaker assumptions; in particular, it is sufficient to assume the map to be semilinear.

We are also interested in more general maps: partial orthometries are maps \linebreak $f \colon X \to Y$ between orthosets that establish an orthoisomorphism, not necessarily between all of $X$ and all of $Y$ but, between a subspace of $X$ and subspace of $Y$. In this case, however, it makes sense to restrict considerations to the narrower class of orthomodular spaces, which we discuss in Section~\ref{sec:Orthomodular-spaces-Dacey-spaces}.

We include at this place a characterisation of orthomodular spaces. It was shown in \cite{LiVe} that the key property for an orthoset $X$ to arise from an orthomodular space is that there is for every subspace $A$ a so-called Sasaki map from $X$ onto $A$. In the present framework, we may describe orthomodular spaces of dimension $\geq 4$ in a similar fashion, namely, as orthosets that are irreducible, fulfil a certain separation property, and are such that any inclusion map of a subspace is adjointable. Adding the condition that for any $x, y \in X$ there is an orthoautomorphism of $X$ sending $x$ to $y$ and keeping every $z \perp x,y$ fixed, we arrive, as in \cite{LiVe}, at a characterisation of the Hilbert spaces over $\Reals$, $\Complexes$, or $\Quaternions$.

Finally, in Section~\ref{sec:Partial-orthometries}, we consider partial orthometries on the one hand and partial quasiisometries between corresponding orthomodular spaces on the other hand. We note that generally, when we correlate maps between Hermitian spaces $H_1$ and $H_2$, and maps between their corresponding orthosets $\P(H_1)$ and $\P(H_2)$, we include a statement concerning the possible usage of linear rather than quasilinear maps. Namely, we may in several cases replace $H_2$ with a space $H_2'$ that shares with $H_1$ the scalar skew field and is equivalent with $H_2$ in the sense that $\P(H_2)$ and $\P(H_2')$ are orthoisomorphic. We achieve in this way that the representing map from $H_1$ to $H_2'$ may be chosen to be linear.

\section{Linear and projective spaces}
\label{sec:Linear-spaces}

The present paper focuses on the description of Hermitian spaces as orthosets and on the question how the respective structure-preserving maps are correlated. A key result needed in this context is Faure and Fr\" olicher's version of the Fundamental Theorem of Projective Geometry (Theorem~\ref{thm:Fundamental-Theorem-Faure-Froelicher} below). This is why we start in this section a review of the situation for linear spaces without inner product. We will on this occasion also fix our notation and we will recall the basic facts about linear spaces, projective spaces, and maps preserving linear dependence. For further details, we refer the reader to the monograph \cite{FaFr}.

\subsubsection*{Linear spaces}

By a {\it sfield}, we will henceforth mean a skew field (i.e., a division ring). We consider linear spaces over arbitrary sfields. We classify only the finite-dimensional spaces with regard to their dimension and we use the symbol $\infty$ to express that the space is infinite-dimensional.

Various types of linearity-preserving maps between linear spaces have been introduced in the literature; we focus on the following version. For a map $\sigma$ between sfields, we will, in accordance with the common practice, denote the image of an element $\alpha$ under $\sigma$ by $\alpha^\sigma$.


\begin{definition} \label{def:semilinear}
Let $V_1$ be a linear space over the sfield $F_1$ and $V_2$ a linear space over the sfield $F_2$. A map $\phi \colon V_1 \to V_2$ is called {\it semilinear} if (i) $\phi(u+v) = \phi(u) + \phi(v)$ for any $u, v \in V_1$ and (ii) there is a homomorphism $\sigma \colon F_1 \to F_2$ such that $\phi(\alpha u) = \alpha^\sigma \phi(u)$ for any $u \in V_1$ and $\alpha \in F_1$. If $\sigma$ is actually an isomorphism, we call $\phi$ {\it quasilinear}.

Moreover, the {\it rank} of a semilinear map $\phi \colon V_1 \to V_2$ is the dimension of the subspace of $V_2$ spanned by the image of $\phi$.
\end{definition}

Unless a semilinear map $\phi$ is the zero map, the homomorphism in condition (ii) of Definition~\ref{def:semilinear} is uniquely determined. We refer to it as the homomorphism {\it associated with $\phi$}.

\subsubsection*{Projective spaces}

We will employ in this paper the slightly modified version of the notion of a projective space proposed in \cite[Remark 6.2.13]{FaFr}. Instead of the collection of all one-dimensional subspaces of some linear space, we consider the collection of all at most one-dimensional subspaces. The harmless modification turns out to be technically convenient.

\begin{definition}
An {\it irreducible projective space with $0$} is a set $P$ equipped with an operation $\star \colon P \times P \to {\mathcal P}(P)$ and a constant $0$, subject to the subsequent conditions. The elements of $P$ distinct from $0$ are called {\it proper}.
\begin{itemize}

\item[\rm (P1)] For any $a, b \in P$, $\{0, a, b\} \subseteq a \star b$ and $a \star b$ contains a further element if and only if $a$ and $b$ are distinct proper elements.

\item[\rm (P2)] If $c$ and $d$ are distinct proper elements of $a \star b$, then $a \star b = c \star d$.

\item[\rm (P3)] Let $a, b, c, d \in P$ be pairwise distinct proper elements. Then $a \star b \;\cap\; c \star d = \{0\}$ if and only if $a \star c \;\cap\; b \star d = \{0\}$.

\end{itemize}
\end{definition}

It is obvious that an irreducible projective space with $0$ may be identified with an irreducible projective space defined in the usual way. The only difference is the presence of the additional element $0$. We shall shorten the cumbersome expression: in the sequel, we will call an irreducible projective space with $0$ simply a projective space.

We adapt the usual notions in the expected way. In particular, a subset $S$ of a projective space $P$ is a {\it subspace} if $0 \in S$ and $a, b \in S$ implies $a \star b \subseteq S$. In this case, $S$ is likewise viewed as a projective space. For $B \subseteq P$, we denote by $\closure B$ the subspace {\it spanned} by $B$, that is, the smallest subspace of $P$ containing $B$. If $P$ is spanned by a finite set, we call the smallest number $n$ such that $P$ is spanned by an $n$-element set the {\it rank} of $P$. If $P$ is not spanned by a finite set, we say that $P$ has infinite rank, or rank $\infty$.

We define structure-preserving maps between projective spaces in the obvious manner. We note that our choice is in accordance with the procedure in \cite{FaFr}; the subsequently defined projective homomorphisms can be identified with Faure and Fr\" olicher's morphisms \cite[Definition 6.2.1]{FaFr}.

\begin{definition} \label{def:homomorphism-of-projective-spaces}
A map $f \colon P_1 \to P_2$ between projective spaces is called a {\it projective homomorphism} if (i)~for any $a, b, c \in P_1$, $a \in b \star c$ implies $f(a) \in f(b) \star f(c)$ and (ii)~$f(0) = 0$. If in this case $f$ is bijective and also $f^{-1}$ is a projective homomorphism, we call $f$ a {\it projective isomorphism}.

Moreover, the {\it rank} of a projective homomorphism $f \colon P_1 \to P_2$ is the rank of the subspace of $P_2$ that is spanned by the image of $f$.
\end{definition}

Projective spaces are based on the concept of linear dependence and below we will consider structures based on an orthogonality relation. The reason for adding in Definition~\ref{def:homomorphism-of-projective-spaces} the attribute ``projective'' is to avoid confusion about which sort of structure we refer to.

\subsubsection*{Linear vs.\ projective spaces}

Disregarding the case of small dimensions, we may say that the transition from linear to projective spaces does not lead to a loss of structure. The coordinatisation theorem reads here as follows. We denote the linear span of vectors $u_1, \ldots, u_k$ by $\lin{u_1, \ldots, u_k}$.

\begin{theorem}
Let $V$ be a linear space over a sfield $F$. We define
\begin{equation} \label{fml:projective-space}
\P(V) \;=\; \{ \lin u \colon u \in H \}
\end{equation}
and for $u, v \in V$, we put
\begin{equation} \label{fml:line-in-linear-space}
\lin u \star \lin v \;=\; \{ \lin w \colon w \in \lin{u,v} \}.
\end{equation}
Then $\P(V)$, equipped with $\star$ and the zero subspace $\{0\}$, is a projective space. The dimension of $V$ coincides with the rank of $\P(V)$.

Conversely, let $P$ be a projective space of rank $\geq 4$. Then there is a linear space $V$ and a projective isomorphism between $P$ and $\P(V)$.
\end{theorem}

Let us now fix a linear space $V_1$ over the sfield $F_1$ and a linear space $V_2$ over the sfield $F_2$. The question is how semilinear maps from $V_1$ to $V_2$ relate to projective homomorphisms from $\P(V_1)$ to $\P(V_2)$.

\begin{proposition} \label{prop:semilinear-induces-homomorphism}
For a semilinear map $\phi \colon V_1 \to V_2$, we may define the map
\begin{equation} \label{fml:map-induced-by-semilinear-map}
\P(\phi) \colon \P(V_1) \to \P(V_2) \komma \lin u \mapsto \lin{\phi(u)},
\end{equation}
and $\P(\phi)$ is a projective homomorphism.
\end{proposition}

\begin{proof}
This is evident.
\end{proof}

Given a semilinear map $\phi$, we call $\P(\phi)$ as defined by~(\ref{fml:map-induced-by-semilinear-map}) the map {\it induced by $\phi$}.

\begin{theorem} \label{thm:Fundamental-Theorem-Faure-Froelicher}
Let $f \colon \P(V_1) \to \P(V_2)$ be a projective homomorphism of rank $\geq 3$. Then $f$ is induced by a semilinear map from $V_1$ to $V_2$.
\end{theorem}

\begin{proof}
See \cite[Theorem 10.1.3]{FaFr}.
\end{proof}

Semilinear maps inducing a certain projective homomorphism are uniquely determined up to a factor.

\begin{lemma} \label{lem:uniqueness-of-semilinear-maps}
Let $\phi \colon V_1 \to V_2$ be a semilinear map of rank $\geq 2$. Let $\psi \colon V_1 \to V_2$ be a further semilinear map. Then $\P(\phi) = \P(\psi)$ if and only if there is a $\kappa \in F_2 \setminus \{0\}$ such that $\psi = \kappa \phi$.

In this case, $\phi$ is quasilinear if and only if so is $\psi$.
\end{lemma}

\begin{proof}
For the main assertion, see \cite[Theorem 6.3.6]{FaFr}.

To show the additional statement, let $\sigma$ be the sfield homomorphism associated with $\phi$ and let $\psi = \kappa \phi$, where $\kappa \in F_2 \setminus \{0\}$. Then $\psi$ is $\sigma'$-linear, where $\alpha^{\sigma'} = \kappa \alpha^\sigma \kappa^{-1}$. We observe that $\sigma$ is an isomorphism if and only if so is $\sigma'$.
\end{proof}

We wonder under which circumstances a projective homomorphism is induced by a quasilinear rather than just a semilinear map. The following facts hold by \cite[10.1.4, 6.5.2, 6.5.5]{FaFr}; we include the short direct proofs.

\begin{lemma} \label{lem:Fundamental-Theorem-Faure-Froelicher-quasilinearity}
Let $\phi \colon V_1 \to V_2$ be a semilinear map. Assume that, for some $2$-dimensional subspaces $S_1$ of \/ $V_1$ and $S_2$ of \/ $V_2$, $\P(\phi)$ maps $\P(S_1)$ onto $\P(S_2)$. Then $\phi$ is quasilinear.
\end{lemma}

\begin{proof}
Let $\sigma \colon F_1 \to F_2$ be the homomorphism associated with $\phi$. We have to show that $\sigma$ is surjective. Let $\{ u, v \}$ be a basis of $S_1$. Then $\{ \phi(u), \phi(v) \}$ is a basis of $S_2$. Let $\delta \in F_2$. Then there are $\alpha, \beta \in F_1$ and $\gamma \in F_2 \setminus \{0\}$ such that $\phi(\alpha u + \beta v) = \gamma(\delta \phi(u) + \phi(v))$. It follows $\alpha^\sigma = \gamma \delta$ and $\beta^\sigma = \gamma$, hence $\delta = (\beta^{-1} \alpha)^\sigma$.
\end{proof}

\begin{theorem} \label{thm:Fundamental-Theorem-Faure-Froelicher-quasilinearity}
Let $f \colon \P(V_1) \to \P(V_2)$ be a projective homomorphism of rank $\geq 3$. The following are equivalent:
\begin{itemize}

\item[\rm (a)] $f$ is induced by a quasilinear map.

\item[\rm (b)] For any $u, v \in V_1$,
\begin{equation} \label{fml:Fundamental-Theorem-Faure-Froelicher-quasilinearity}
f(\lin u \star \lin v) \;=\; f(\lin u) \star f(\lin v).
\end{equation}
\item[\rm (c)] There are non-zero vectors $u, v \in V_1$ such that $f(\lin u) \neq f(\lin v)$ and~{\rm (\ref{fml:Fundamental-Theorem-Faure-Froelicher-quasilinearity})} holds.

\end{itemize}
\end{theorem}

\begin{proof}
By Theorem~\ref{thm:Fundamental-Theorem-Faure-Froelicher}, $f = \P(\phi)$ for some semilinear map $\phi$. We readily check that (a) implies (b), and (b) implies trivially (c). Furthermore, (c) implies (a) by Lemma~\ref{lem:Fundamental-Theorem-Faure-Froelicher-quasilinearity}.
\end{proof}

\begin{corollary} \label{cor:Fundamental-Theorem-Faure-Froelicher-1}
Let $f \colon \P(V_1) \to \P(V_2)$ be a projective homomorphism. Assume that there are at least $3$-dimensional subspaces $Z_1$ of $V_1$ and $Z_2$ of $V_2$ such that $f$ establishes a projective isomorphism between $\P(Z_1)$ and $\P(Z_2)$. Then $f$ is induced by a quasilinear map.
\end{corollary}

\begin{proof}
This is clear from Theorem~\ref{thm:Fundamental-Theorem-Faure-Froelicher-quasilinearity}.
\end{proof}

The following corollary contains the (common version of the) Fundamental Theorem of Projective Geometry; see, e.g., {\rm \cite[Section III.1]{Bae}}. 

\begin{corollary} \label{cor:Fundamental-Theorem-Faure-Froelicher-2}
Let $V_1$ and $V_2$ be at least $3$-dimensional and let $f \colon \P(V_1) \to \P(V_2)$ a bijection such that~{\rm (\ref{fml:Fundamental-Theorem-Faure-Froelicher-quasilinearity})} holds for any $u, v \in V_1$. Then $f$ is induced by a quasilinear bijection.
\end{corollary}

\begin{proof}
We have that $f(\lin 0) = \lin 0$, because $\{ f(\lin 0) \} = f(\{\lin 0\}) = f(\lin 0 \star \lin 0) = f(\lin 0) \star f(\lin 0) = \{f(\lin 0), \lin 0\}$. Hence $f$ is a projective homomorphism and by Theorem~\ref{thm:Fundamental-Theorem-Faure-Froelicher-quasilinearity}, $f$ is induced by a quasilinear map $\phi$. As $\P(\phi)$ is bijective, so is $\phi$.
\end{proof}

We can clearly make a quasilinear map into a linear one by replacing the scalar sfield of the codomain space with an isomorphic one. We provide the details for later reference.

\begin{lemma} \label{lem:quasilinear-to-linear}
Let $\phi \colon V_1 \to V_2$ be a quasilinear map and let $f = \P(\phi)$. Then there is a linear space $V_2'$ over $F_1$ and a projective isomorphism $t \colon \P(V_2) \to \P(V_2')$ such that $t \circ f$ is induced by a linear map.
\end{lemma}

\begin{proof}
Let $\sigma \colon F_1 \to F_2$ be the isomorphism associated with $\phi$. We redefine the linear structure of $V_2$ as follows: the addition of vectors remains unmodified and the scalar multiplication is defined by $\alpha \scmu u = \alpha^\sigma u$ for $\alpha \in F_1$ and $u \in V_2$. Let us call the resulting linear space $V_2'$.

We observe that $\tau \colon V_2 \to V_2' \komma u \mapsto u$ is $\sigma^{-1}$-linear and consequently $\tau \circ \phi$ is linear. Moreover, $t = \P(\tau)$ is a projective isomorphism. Finally, $t \circ f = \P(\tau \circ \phi)$.
\end{proof}

\section{Hermitian spaces and linear orthosets}
\label{sec:Hermitian-spaces}

So far we have considered semilinear maps on the one hand and projective homomorphisms on the other hand. We shall now consider the case when the linear spaces are equipped with an inner product. The question arises whether the correspondence described in Proposition~\ref{prop:semilinear-induces-homomorphism} and Theorem~\ref{thm:Fundamental-Theorem-Faure-Froelicher} for linear spaces possesses a version for Hermitian spaces. We shall see that this is indeed the case, provided that the structure-preserving maps are suitably chosen.

\subsubsection*{Hermitian spaces}

By a {\it $\ast$-sfield}, we mean a sfield $F$ equipped with an involutory antiautomorphism $\ast \colon F \to F$. Let $H$ be a linear space over the $\ast$-sfield $F$. A \emph{Hermitian form} on $H$ is a symmetric $\ast$-sesquilinear form, that is, a map $\herm{\cdot}{\cdot} \colon H \times H \to F$ such that, for any $u, v, w \in H$ and $\alpha, \beta \in F$,
\begin{align*}
& \herm{\alpha u + \beta v}{w} \;=\; \alpha \herm{u}{w} + \beta \herm{v}{w}, \\
& \herm{w}{\alpha u + \beta v} \;=\;
                               \herm{w}{u} \alpha^\ast + \herm{w}{v} \beta^\ast, \\
& \herm{u}{v} \;=\; \herm{v}{u}^\ast.
\end{align*}
We additionally assume all Hermitian forms to be \emph{anisotropic}, that is, $\herm{u}{u} = 0$ implies $u = 0$. Endowed with an (anisotropic) Hermitian form, $H$ is referred to as a \emph{Hermitian space} over $F$.

A Hermitian space $H$ comes with a natural orthogonality relation: for two vectors $u, v \in H$, we put $u \perp v$ if $\herm u v = 0$. We tacitly assume that there are at most countably many mutually orthogonal vectors.

\begin{lemma} \label{lem:orthogonal-basis}
Any finite-dimensional Hermitian space possesses an orthogonal basis.
\end{lemma}

\begin{proof}
Let $\{ b_1, \ldots, b_n \}$ be a basis of the Hermitian space $H$. For $k = 1, \ldots, n$, we successively define $e_k = b_k - \sum_{i=1}^{k-1} \herm{b_k}{e_i} \herm{e_i}{e_i}^{-1} e_i$. We readily check that $\{ e_1, \ldots, e_n \}$ is an orthogonal basis.
\end{proof}

For any $A \subseteq H$, $A\c = \{ u \in H \colon u \perp v \text{ for all } v \in A \}$ is a linear subspace of $H$. The assignment $A \mapsto A\cc$ defines a closure operation on $H$ and a subspace $S$ of $H$ such that $S = S\cc$ is called {\it orthoclosed}. We note that, for a finite set $A \subseteq H$, $A\cc$ is the linear span of $A$. Furthermore, a subspace $S$ of $H$ is called {\it splitting} if $H$ is the direct sum of $S$ and $S\c$. Any splitting subspace is orthoclosed but the converse does not hold in general.

\begin{lemma} \label{lem:splitting-subspaces}
Let $S$ be a finite-dimensional subspace of a Hermitian space $H$. Then $S$ is splitting.
\end{lemma}

\begin{proof}
Let $u \in H$. By Lemma \ref{lem:orthogonal-basis}, $S$ has an orthogonal basis $\{e_1, \ldots, e_n \}$. Let $u_S = \sum_{i=1}^{n} \herm{u}{e_i} \herm{e_i}{e_i}^{-1} e_i$ and $u_{S\c} = u - u_S$. Then $u_S \in S$ and $u_{S\c} \in S\c$, and we have $u = u_S + u_{S\c}$.
\end{proof}

A reason to equip a linear space with a non-degenerate inner product might be to provide an identification of the space with its dual. In the present context the situation is as follows.

Let $H$ be a Hermitian space over $F$. Its algebraic dual space is the (right) linear space consisting of all linear maps from $H$ to $F$. For the present purposes, this space is in general too large. We put
\[ H^\ast \;=\; \{ \rho \colon H \to F \;\colon\; \text{$\rho$ is linear} \text{ and } \text{$\kernel \rho$ is orthoclosed} \} \]
and we refer to $H^\ast$ simply as the {\it dual space} of $H$. The following lemma implies that $H^\ast$ is a subspace of the algebraic dual space.

\begin{lemma} \label{lem:dual-space}
For any Hermitian space $H$, we have $H^\ast = \{ \herm{\cdot}{u} \colon u \in H \}$ and
\begin{equation} \label{fml:Hstar}
H \to H^\ast \komma u \mapsto \herm{\cdot}{u}
\end{equation}
is a bijection.
\end{lemma}

\begin{proof}
We clearly have $\herm{\cdot}{u} \in H^\ast$ for any $u \in H$. Conversely, let $\rho \in H^\ast$ and assume $\rho \neq 0$. Then there is an $x \perp \kernel \rho$ such that $\rho(x) = 1$. As $y - \rho(y) x \in \kernel \rho$ for any $y \in H$, we have $\herm{y}{x} = \herm{y - \rho(y) x + \rho(y) x}{\,x} = \rho(y) \herm{x}{x}$ and thus $\rho = \herm{\cdot}{\herm{x}{x}^{-1} x}$.
\end{proof}

\begin{example}
1. For a finite-dimensional Hermitian space $H$, all subspaces are orthoclosed. Hence $H^\ast$ coincides with the algebraic dual space.

2. Let $H$ be a complex Hilbert space. Then a linear form $\rho \colon H \to \Complexes$ is in $H^\ast$ if and only if $\rho$ is bounded (i.e., continuous).
\end{example}

We shall now specify the maps between Hermitian spaces that we consider as relevant. The preservation of linear dependence is not a sufficient condition. The preservation of the Hermitian form is in turn a quite restrictive condition. Let $\phi \colon H_1 \to H_2$ be a linear map between Hermitian spaces. We will require that the algebraic adjoint of $\phi$ restricts to a map from $H_2^\ast$ to $H_1^\ast$.

\begin{definition} \label{def:adjoint-of-linear-map}
Let $\phi \colon H_1 \to H_2$ be a linear map between Hermitian spaces over the same $\ast$-sfield. We call $\phi$ {\it adjointable} if $\rho \circ \phi \in H_1^\ast$ for any $\rho \in H_2^\ast$. In this case, the {\it adjoint} of $\phi$, denoted by $\phi\adj$, is the map $\rho \mapsto \rho \circ \phi$ under the identification~(\ref{fml:Hstar}) of the Hermitian spaces with their duals.
\end{definition}

In other words, for $\phi \colon H_1 \to H_2$ to possess the adjoint $\phi\adj \colon H_2 \to H_1$ means that
\begin{equation} \label{fml:adjoint-of-linear-map}
\herm{\phi(u)}{v} \;=\; \herm{u}{\phi\adj(v)}
\quad\text{for any $u \in H_1$ and $v \in H_2$.}
\end{equation}
We observe from~(\ref{fml:adjoint-of-linear-map}) that $\phi\adj$ is linear. Besides, the symmetry of the Hermitian forms implies that $\phi\adj$ is adjointable as well, having the adjoint $\phi$.

\begin{example} \label{ex:adjointable-maps}
1. Let $H_1$ and $H_2$ be finite-dimensional Hermitian spaces. Then any linear map $\phi \colon H_1 \to H_2$ is adjointable.

2. Let $H_1$ and $H_2$ be complex Hilbert spaces. Then a linear map $\phi \colon H_1 \to H_2$ is adjointable if and only if $\phi$ is bounded. The ``only if'' part of this equivalence is a consequence of the closed graph theorem; see, e.g., {\rm \cite[Proposition 2.3.11]{Ped}}.
\end{example}

We note that the notion of an adjoint map is available also for maps that are quasilinear rather than linear. In the present context, however, the larger generality does not seem to be much of an advantage.

\subsubsection*{Orthosets}

We proceed by considering the projective structure associated with a Hermitian space $H$. It might seem natural to associate with $H$ an orthogeometry, consisting of the projective space $\P(H)$ and the orthogonality relation $\perp$ \cite[Section 14.1]{FaFr}. However, the following lemma shows that, given the orthogonality relation, the relation of linear dependence is already uniquely determined. We may hence work with a structure that is based solely on an orthogonality relation.

\begin{lemma} \label{lem:perp-determines-star}
For vectors $u$, $v$ of a Hermitian space, $\{ u, v \}\cc = \lin{u, v}$. Consequently,
\[ \{ \lin u, \lin v \}\cc \;=\; \lin u \star \lin v. \]
\end{lemma}

\begin{proof}
By Lemma~\ref{lem:splitting-subspaces}, the $2$-dimensional subspace $\lin{u,v}$ is splitting and hence orthoclosed.
\end{proof}

\begin{definition}
An {\it orthoset} is a set $X$ equipped with a binary relation $\perp$ and a constant $0$ such that the following holds:
\begin{itemize}

\item[\rm (O1)] $x \perp y$ implies $y \perp x$ for any $x, y \in X$,

\item[\rm (O2)] $x \perp x$ if and only if $x = 0$,

\item[\rm (O3)] $0 \perp x$ for any $x \in X$.

\end{itemize}
Elements $x, y \in X$ such that $x \perp y$ are called {\it orthogonal}. The element $0$ is called {\it falsity} and the elements distinct from $0$ are called {\it proper}.

Finally, a {\it $\perp$-set} is a subset of $X$ consisting of mutually orthogonal proper elements. If, for some $n \in \Naturals$, $X$ contains an $n$-element $\perp$-set and any other $\perp$-set contained in $X$ has at most size $n$, we call $n$ the {\it rank} of $X$. If there is no such $n$, we say that $X$ has infinite rank, or rank $\infty$.
\end{definition}

Given an orthoset $X$, we make the following additional definitions. For a subset $A$ of an orthoset $X$, we put $A\c = \{ x \in X \colon x \perp y \text{ for every $y \in A$} \}$. We call $A$ {\it orthoclosed} if $A\cc = A$. In this case, $A$, equipped with the orthogonality relation inherited from $X$, is likewise an orthoset, which we call a {\it subspace} of $X$.

The following kind of maps will play the role of structure-preserving maps between orthosets. For a further discussion, we refer to \cite{PaVe4}. For an alternative approach, which puts the focus on the preservation of the orthogonality relation, see \cite{PaVe1,PaVe2}.

In what follows, the image of (whatever kind of) a map $f$ is denoted by $\image f$.

\begin{definition}
A $f \colon X \to Y$ between orthosets is called {\it adjointable} if there is a map $g \colon Y \to X$ such that, for any $x \in X$ and $y \in Y$,
\[ f(x) \perp y \text{ if and only if } x \perp g(y). \]
In this case, $g$ is called an {\it adjoint} of $f$.

Moreover, the {\it rank} of $f$ is the rank of the subspace $(\image f)\cc$ of $Y$.
\end{definition}

Clearly, if $g$ is an adjoint of $f$, then also $f$ is an adjoint of $g$, and hence we occasionally refer to $f$ and $g$ as an {\it adjoint pair}.

Our discussion is focused on {\it irredundant} orthosets. Irredundancy means that any elements $x$ and $y$ such that $\{x\}\c = \{y\}\c$ coincide.

\begin{lemma} \label{lem:unique-adjoint}
Let $f \colon X \to Y$ be an adjointable map and assume that $X$ is irredundant. Then $f$ possesses a unique adjoint.
\end{lemma}

\begin{proof}
For any adjoint $g \colon Y \to X$ of $f$, the orthocomplement of $g(y)$, $y \in Y$, is uniquely determined by $f$: we have $\{g(y)\}\c = \{ x \colon f(x) \perp y \}$.
\end{proof}

If a map $f$ between orthosets has a unique adjoint, we will usually denote it by $f\adj$.

Adjointable maps preserve the membership of an element in the closure of a subset. Maps between closure spaces with this property are called {\it continuous}; see, e.g., \cite[Section~2.1]{Ern}.

\begin{lemma} \label{lem:continuity-of-adjointable-maps}
Let $f \colon X \to Y$ be an adjointable map between orthosets. Then the following holds.
\begin{itemize}

\item[\rm (i)] $f(0) = 0$.

\item[\rm (ii)] For any $A \subseteq X$, we have $f(A\cc) \subseteq f(A)\cc$. In fact, $f(A\cc)\cc = f(A)\cc$.

\item[\rm (iii)] If $A \subseteq Y$ is orthoclosed, so is $f^{-1}(A)$.

\end{itemize}
\end{lemma}

\begin{proof}
See \cite[Lemma 3.5]{PaVe4}.
\end{proof}

The continuity of an adjointable map between orthosets implies a property that might be interpreted as the preservation of linear dependence: if $f \colon X \to Y$ is adjointable, we have by Lemma~\ref{lem:continuity-of-adjointable-maps}(ii) for any $x_1, x_2 \in X$,
\begin{equation} \label{fml:preserving-linear-dependence}
f(\{ x_1, x_2 \}\cc) \;\subseteq\; \{ f(x_1), f(x_2) \}\cc.
\end{equation}

In order to express the structural coincidence of two orthosets, we will use the following notion.

\begin{definition}
We call a bijection $f \colon X \to Y$ between orthosets an {\it orthoisomor\-phism} if, for any $x, y \in X$, $x \perp y$ holds if and only if $f(x) \perp f(y)$.
\end{definition}

\begin{proposition} \label{prop:orthoisomorphism}
Let $f \colon X \to Y$ be a bijection between orthosets. Then $f$ is an orthoisomorphism if and only if $f$ and $f^{-1}$ form an adjoint pair.
\end{proposition}

\begin{proof}
The map $f$ is an orthoisomorphism iff, for any $x, y \in X$, $x \perp y$ is equivalent to $f(x) \perp f(y)$ iff, for any $x \in X$ and $z \in Y$, $f(x) \perp z$ is equivalent to $x \perp f^{-1}(z)$.
\end{proof}

\subsubsection*{Hermitian spaces vs.\ linear orthosets}

Provided that dimensions are at least $4$, Hermitian spaces correspond to orthosets of a certain kind, which we now specify.

\begin{definition}
An orthoset $X$ is called {\it linear} if, for any distinct proper elements $x, y \in X$, there is a proper element $z \in X$ such that $\{x,y\}\c = \{x,z\}\c$ and exactly one of $y$ and $z$ is orthogonal to $x$.
\end{definition}

Note that any linear orthoset is irredundant. Indeed, if $x \neq y$ and $x \notperp y$, there is a proper $z \perp x$ such that $\{x,y\}\c = \{x,z\}\c$. But then $z \perp y$ would imply $z \in \{x,y\}\c$ and hence $z \perp z$. Thus $z \in \{x\}\c$ but $z \notin \{y\}\c$.

\begin{theorem} \label{thm:Hermitian-spaces-linear-orthosets}
Let $H$ be a Hermitian space. We equip $\P(H)$ with the relation $\perp$ given by
\[ \lin u \perp \lin v \quad\text{if}\quad u \perp v \]
for $u, v \in H$, and with the zero subspace $\{0\}$. Then $\P(H)$ is a linear orthoset, whose rank coincides with the dimension of $H$.

Conversely, let $X$ be a linear orthoset of rank $\geq 4$. Then there is a Hermitian space $H$ and an orthoisomorphism between $X$ and $\P(H)$.
\end{theorem}

\begin{proof}
See \cite[Theorem 4.5]{PaVe3}.
\end{proof}

We turn to the question how adjointable maps between Hermitian spaces relate to adjointable maps between the corresponding orthosets.

\begin{proposition} \label{prop:adjointable-linear-map}
Let $\phi \colon H_1 \to H_2$ be an adjointable linear map between Hermitian spaces. Then $\P(\phi)$ is adjointable and its adjoint is $\P(\phi)\adj = \P(\phi\adj)$.
\end{proposition}

\begin{proof}
For any $u \in H_1$ and $v \in H_2$, we have $\phi(u) \perp v$ iff $\herm{\phi(u)}{v} = 0$ iff $\herm{u}{\phi\adj(v)} = 0$ iff $u \perp \phi\adj(v)$. It follows that $\P(\phi\adj)$ is an adjoint of $\P(\phi)$. As $\P(H_1)$ is irredundant, the adjoint is by Lemma~\ref{lem:unique-adjoint} unique.
\end{proof}

For the converse direction, some preparations are necessary. For the remainder of the section, let $H_1$ be a Hermitian space over the sfield $F_1$ and $H_2$ a Hermitian space over the sfield $F_2$.

\begin{lemma} \label{lem:adjointable-map-is-homomorphism}
Any adjointable map $f \colon \P(H_1) \to \P(H_2)$ between orthosets is a projective homomorphism. 

Moreover, the rank of $f$ as a map between orthosets coincides with the rank of $f$ as a projective homomorphism.
\end{lemma}

\begin{proof}
For $u, v, w \in H_1$, $\,\lin u \in \{ \lin v, \lin w \}\cc$ implies $f(\lin u) \in \{ f(\lin v), f(\lin w) \}\cc$ by~(\ref{fml:preserving-linear-dependence}), and we have $f(\{0\}) = \{0\}$ by Lemma~\ref{lem:continuity-of-adjointable-maps}(i). In view of Lemma~\ref{lem:perp-determines-star}, the first assertion follows.

It remains to show that the rank $k$ of the subspace $\closure \image f$ of the projective space $\P(H_2)$ coincides with the rank $k'$ of the subspace $(\image f)\cc$ of the orthoset $\P(H_2)$. Let $S$ be the linear subspace of $H_2$ spanned by the union of the elements of $\image f$. Then $\P(S) = \closure \image f$ and $S$ is $k$-dimensional. Assume that $k$ is finite. Then $S$ is orthoclosed, hence so is $\P(S)$. As $\closure \image f \subseteq (\image f)\cc$, we conclude that $\P(S) = (\image f)\cc$. Moreover, by Lemma~\ref{lem:orthogonal-basis}, $S$ has a $k$-element orthogonal basis and hence the rank of $\P(S)$ is $k$. That is, $k' = k$. If $k = \infty$, $\closure \image f$ contains an infinite set of mutually orthogonal one-dimensional subspaces. Hence $k' = \infty$ as well.
\end{proof}

\begin{lemma} \label{lem:rank-of-adjoint}
Let $f \colon \P(H_1) \to \P(H_2)$ be adjointable. Then the ranks of $f$ and its adjoint $f\adj$ coincide.
\end{lemma}

\begin{proof}
Assume first that $f$ has the finite rank $k$. Again, the subspace $S$ of $H_2$ spanned by the union of the elements of $\image f$ is $k$-dimensional and we have $\P(S) = (\image f)\cc$. Let $\{ u_1, \ldots, u_k \}$ be a basis of $S$ and let $B = \{ \lin{u_1}, \ldots, \lin{u_k} \}$. Then $B\cc = (\image f)\cc$ and by \cite[Lemma 3.5, 3.8]{PaVe4}, $(\image f\adj)\cc = f\adj((\image f)\cc)\cc = \linebreak f\adj(B\cc)\cc = f\adj(B)\cc$. Hence $\image f\adj$ is contained in an at most $k$-dimensional subspace of $H_1$ and we conclude that $f\adj$ has a rank $\leq k$. In particular, also $f\adj$ has a finite rank, and by symmetry it follows that both ranks coincide.

If $f$ has infinite rank, it is now clear that $f\adj$ has infinite rank as well.
\end{proof}

\begin{lemma} \label{lem:inclusion-of-finite-dimensional-subspaces}
Let $S$ be a finite-dimensional subspace of a Hermitian space $H$. Then the inclusion map $\iota \colon S \to H$ is adjointable and the adjoint is the orthogonal projection of $H$ to $S$.
\end{lemma}

\begin{proof}
By Lemma~\ref{lem:splitting-subspaces}, every finite-dimensional subspace of a Hermitian space is splitting.
\end{proof}

\begin{lemma} \label{lem:adjointable-quasilinear}
Let $\phi \colon H_1 \to H_2$ and $\psi \colon H_2 \to H_1$ be semilinear maps such that $\P(\phi)$ and $\P(\psi)$ form an adjoint pair. Assume furthermore that $\phi$ has rank $\geq 2$. Then $\phi$ is quasilinear.
\end{lemma}

\begin{proof}
Let $S_1$ be a $2$-dimensional subspace of $H_1$ such that the subspace $S_2$ of $H_2$ spanned by $\phi(S_1)$ is $2$-dimensional. By Lemma~\ref{lem:inclusion-of-finite-dimensional-subspaces}, the inclusion maps $\iota_1 \colon S_1 \to H_1$ and $\iota_2 \colon S_2 \to H_2$ have adjoints $\pi_1$ and $\pi_2$, which are the orthogonal projections. Put $\tilde\phi = \pi_2 \circ \phi \circ \iota_1$ and $\tilde\psi = \pi_1 \circ \psi \circ \iota_2$. Then $\P(\tilde\phi) \colon \P(S_1) \to \P(S_2)$ and $\P(\tilde\psi) \colon \P(S_2) \to \P(S_1)$ are an adjoint pair: for any $u \in S_1$ and $v \in S_2$, we have $\tilde\phi(u) \perp v$ if and only if $u \perp \tilde\psi(v)$.

We shall show that $\P(\tilde\phi)$ is surjective; then the assertion will follow from Lemma~\ref{lem:Fundamental-Theorem-Faure-Froelicher-quasilinearity}. Let $v \in S_2 \setminus \{0\}$. Let $v' \in S_2 \setminus \{0\}$ be such that $v' \perp v$, and let $u \in S_1 \setminus \{0\}$ be such that $u \perp \tilde\psi(v')$ and hence $\tilde\phi(u) \perp v'$. If $\tilde\phi(u) = 0$, it would follow that $\phi(S_1) = \tilde\phi(S_1)$ is at most $1$-dimensional, a contradiction. Hence $\tilde\phi(u) \neq 0$ and we conclude $\P(\tilde\phi)(\lin u) = \lin{\tilde\phi(u)} = \lin{v}$.
\end{proof}

\begin{theorem} \label{thm:adjointable-quasilinear}
Let $f \colon \P(H_1) \to \P(H_2)$ be an adjointable map of rank $\geq 3$. Then $f$ is induced by a quasilinear map from $H_1$ to $H_2$.
\end{theorem}

\begin{proof}
By Lemma~\ref{lem:rank-of-adjoint}, also its adjoint $f\adj$ has rank $\geq 3$. By Lemma~\ref{lem:adjointable-map-is-homomorphism}, $f$ and $f\adj$ are projective homomorphisms of rank $\geq 3$. By Theorem~\ref{thm:Fundamental-Theorem-Faure-Froelicher}, there are semilinear maps $\phi \colon H_1 \to H_2$ and $\psi \colon H_2 \to H_1$ such that $f = \P(\phi)$ and $f\adj = \P(\psi)$. By Lemma~\ref{lem:adjointable-quasilinear}, $\phi$ is quasilinear.
\end{proof}

Given a quasilinear map $\phi \colon H_1 \to H_2$, we may again replace the scalar sfield of $H_2$ to achieve linearity. We note that the space $H_2'$ replacing $H_2$ will share with $H_1$ the same scalar sfield, but the involutions might not coincide.

\begin{lemma} \label{lem:quasilinear-to-linear-for-Hermitian-spaces}
Let $\phi \colon H_1 \to H_2$ be a quasilinear map and let $f = \P(\phi)$. Then there is a Hermitian space $H_2'$ over $F_1$ and an orthoisomorphism $t \colon \P(H_2) \to \P(H_2')$ such that $t \circ f$ is induced by a linear map from $H_1$ to $H_2'$.
\end{lemma}

\begin{proof}
Let $\sigma \colon F_1 \to F_2$ be the isomorphism associated with $\phi$ and let $H_2'$ be the linear space over $F_1$ that is constructed from $H_2$ as in the proof of Lemma~\ref{lem:quasilinear-to-linear}. Let $\sigma \colon F_1 \to F_2$ be the isomorphism associated with $\phi$ and let $H_2'$ be the linear space over $F_1$ that is constructed from $H_2$ as in the proof of Lemma~\ref{lem:quasilinear-to-linear}. Moreover, for $u, v \in H_2$ let $\herm{u}{v}_2' = {\herm{u}{v}_2}^{\sigma^{-1}}$. We readily check that $\herm{\cdot}{\cdot}_2'$ is a Hermitian form on $H_2'$ w.r.t.\ to the involution $\ast_2'$ given by $\alpha^{\ast_2'} = \alpha^{\sigma \ast_2 \sigma^{-1}}$.

Furthermore, $\tau \colon H_2 \to H_2' \komma u \mapsto u$ is $\sigma^{-1}$-linear and hence $\tau \circ \phi$ is linear. Clearly, $t = \P(\tau)$ is an orthoisomorphism between $\P(H_2)$ and $\P(H_2')$, and we have $t \circ f = \P(\tau \circ \phi)$.
\end{proof}

A linear map inducing an adjointable map between orthosets is likewise adjointable.

\begin{theorem} \label{thm:phi-Pphi-adjointable}
Let $\phi \colon H_1 \to H_2$ be a linear map. Then $\phi$ is adjointable if and only if \/ $\P(\phi)$ is adjointable.
\end{theorem}

\begin{proof}
The ``only if'' part holds by Proposition~\ref{prop:adjointable-linear-map}.

Before proceeding, we note that any Hermitian space $H$, equipped with the usual orthogonal relation and the zero vector, is an orthoset. Moreover, the orthoclosed subspaces of the Hermitian space $H$ coincide with the orthoclosed subsets of the orthoset $H$.

To see the ``if'' part, assume that $\P(\phi)$ is adjointable. It follows that $\phi$, seen as a map between orthosets, is likewise adjointable. For any $\rho \in {H_2}^\ast$, $\kernel \rho$ is by definition an orthoclosed subspace of $H_2$ and hence an orthoclosed subset of the orthoset $H_2$. By Lemma~\ref{lem:continuity-of-adjointable-maps}(iii), $\kernel (\rho \circ \phi) = \phi^{-1}(\kernel \rho)$ is an orthoclosed subset of the orthoset $H_1$. This in turn means that $\kernel (\rho \circ \phi)$ is an orthoclosed subspace of the Hermitian space $H_1$ and hence $\rho \circ \phi \in {H_1}^\ast$. We conclude that $\phi$ is adjointable.
\end{proof}

\begin{remark}
In the particular case of complex Hilbert spaces, we have by Theorem~\ref{thm:phi-Pphi-adjointable} and Example~\ref{ex:adjointable-maps}-2 that, for any linear map $\phi \colon H_1 \to H_2$, the following are equivalent: {\rm (a)} $\phi$ is adjointable, {\rm (b)} $\P(\phi)$ is adjointable, {\rm (c)} $\phi$ is bounded.
\end{remark}

\section{Orthometric correspondences}
\label{sec:Orthometric-correspondences}

We have dealt in the preceding section with adjointable maps between pairs of Hermitian spaces on the one hand and between pairs of the corresponding orthosets on the other hand. In this section, we focus on a certain type of such maps. In case of the Hermitian spaces, we consider bijections that preserve, in some sense, the Hermitian form. In case of the orthosets, we consider orthogonality-preserving bijections.

\begin{definition} \label{def:quasiunitary}
Let $H_1$ be a Hermitian space over the sfield $F_1$ and $H_2$ a Hermitian space over the sfield $F_2$. We call a bijection $\phi \colon H_1 \to H_2$ {\it quasiunitary} if $\phi$ is quasilinear with associated isomorphism $\sigma$ and there is a $\lambda \in F_2 \setminus \{0\}$ such that, for any $u, v \in H_1$,
\begin{equation} \label{fml:quasiunitary}
\herm{\phi(u)}{\phi(v)}_2 \;=\; {\herm{u}{v}_1}^\sigma \; \lambda.
\end{equation}
If in this situation $F_1$ and $F_2$ coincide, $\sigma$ is the identity, and $\lambda = 1$, $\phi$ is called {\it unitary}. For $H_1$ and $H_2$ to be {\it isomorphic} means that there is a unitary map $\phi \colon H_1 \to H_2$.
\end{definition}

\begin{remark} \label{rem:quasiunitary}
Assume that $H_1$ and $H_2$ are not the zero spaces, and let the scalar sfields $F_1$ and $F_2$ be equipped with the involutions $\ast_1$ and $\ast_2$, respectively. If there is a quasiunitary map $\phi \colon H_1 \to H_2$, we conclude from~{\rm (\ref{fml:quasiunitary})} that $\ast_1$ and $\ast_2$ are related by the equation $\alpha^{\sigma \ast_2} = \lambda^{-1} \alpha^{\ast_1 \sigma} \lambda$, $\;\alpha \in F_1$. In particular, if $\phi$ is unitary, we have that $\ast_1 = \ast_2$.

We note that furthermore $\lambda^{\ast_2} = \lambda$. Indeed, let $u \in H_1 \setminus \{0\}$. Then $\herm{\phi(u)}{\phi(u)}_2 = {\herm{\phi(u)}{\phi(u)}_2}^{\ast_2}$ and hence, by~{\rm (\ref{fml:quasiunitary})},
${\herm{u}{u}_1}^\sigma \lambda
= \lambda^{\ast_2} {\herm{u}{u}_1}^{\sigma\ast_2}
= \lambda^{\ast_2} \lambda^{-1} {\herm{u}{u}_1}^{\ast_1\sigma} \lambda
= \lambda^{\ast_2} \lambda^{-1} {\herm{u}{u}_1}^\sigma \lambda$ and the claim follows.
\end{remark}

Among the linear maps, we may characterise unitary maps by means of adjointability, cf.\ Proposition~\ref{prop:orthoisomorphism}.

\begin{proposition} \label{prop:unitary-map}
Let $\phi \colon H_1 \to H_2$ be a bijective linear map between Hermitian spaces. Then $\phi$ is unitary if and only if $\phi$ and $\phi^{-1}$ form an adjoint pair.
\end{proposition}

\begin{proof}
The map $\phi$ is unitary iff $\herm{\phi(u)}{\phi(v)}_2 = \herm{u}{v}_1$ for any $u, v \in H_1$ iff $\herm{\phi(u)}{w}_2 = \herm{u}{\phi^{-1}(w)}_1$ for any $u \in H_1$ and $w \in H_2$.
\end{proof}

In line with our previous considerations, we wonder how (quasi)unitary maps between Hermitian spaces are related to orthoisomorphisms of the corresponding orthosets.

\begin{proposition} \label{prop:quasiunitary-map-induces-orthoiso}
Let $\phi \colon H_1 \to H_2$ be a quasiunitary map between Hermitian spaces. Then $\P(\phi)$ is an orthoisomorphism.
\end{proposition}

\begin{proof}
The map $\P(\phi)$ is bijective because so is $\phi$. Hence the assertion is immediate from~(\ref{fml:quasiunitary}).
\end{proof}

For the converse direction, the key result of which we may make use is a generalisation of a theorem of Piziak \cite[Theorem (2.2)]{Piz1}. Piziak proved the subsequent result under the additional assumption that the sesquilinear forms are orthosymmetric and that $\phi$ is quasilinear. The hypothesis of orthosymmetry is not used in his proof. In order to weaken the hypothesis of quasilinearity, however, we need to modify the argumentation. For the reader's convenience, we reproduce shortly also the remaining parts of the proof.

\begin{theorem} \label{thm:Piziak-original}
Let $H_1$ be an at least $2$-dimensional linear space over $F_1$, equipped with a non-degenerate sesquilinear form $\herm{\cdot}{\cdot}_1$, and let $H_2$ be a linear space over $F_2$ equipped with a sesquilinear form $\herm{\cdot}{\cdot}_2$. Let $\phi \colon H_1 \to H_2$ be a semilinear map with associated homomorphism $\sigma$ and assume that $u \perp v$ implies $\phi(u) \perp \phi(v)$ for any $u, v \in H_1$. Then there is a unique $\lambda \in F_2$ such that~{\rm (\ref{fml:quasiunitary})} holds for any $u, v \in H_1$.
\end{theorem}

\begin{proof}
Let $v \in H_1 \setminus \{0\}$. As $\herm{\cdot}{\cdot}_1$ is supposed to be non-degenerate, there is a $w \in H_1$ such that $\herm{w}{v}_1 = 1$. Put $\lambda_v = \herm{\phi(w)}{\phi(v)}_2$. We claim that
\[ \herm{\phi(u)}{\phi(v)}_2
\;=\; {\herm{u}{v}_1}^\sigma \, \lambda_v \quad \text{for all $u \in H_1$.} \]
Indeed, $u-\herm{u}{v}_1 w \perp v$ because $\herm{u-\herm{u}{v}_1 w}{\,v}_1 = \herm{u}{v}_1 - \herm{u}{v}_1 \herm{w}{v}_1 = 0$. Hence $\phi(u-\herm{u}{v}_1 w) \perp \phi(v)$, that is,
\[ \begin{split}
0 & \;=\; \herm{\phi(u-\herm{u}{v}_1 w)}{\phi(v)}_2
\;=\; \herm{\phi(u)}{\phi(v)}_2 - {\herm{u}{v}_1}^\sigma \herm{\phi(w)}{\phi(v)}_2 \\
& \;=\; \herm{\phi(u)}{\phi(v)}_2 - {\herm{u}{v}_1}^\sigma \lambda_v. \end{split}\]
It is our aim to show that the factor $\lambda_v$ is independent of $v$. Let $v' \in H_1 \setminus \{0\}$ be linearly independent of $v$. Then for any $u \in H_1$
\[ \begin{split}
& {\herm{u}{v}_1}^\sigma \, \lambda_v + {\herm{u}{v'}_1}^\sigma \, \lambda_{v'}
\;=\; \herm{\phi(u)}{\phi(v)}_2 + \herm{\phi(u)}{\phi(v')}_2
\;=\; \herm{\phi(u)}{\phi(v+v')}_2 \\
& =\; {\herm{u}{v+v'}_1}^\sigma \, \lambda_{v+v'}
\;=\; {\herm{u}{v}_1}^\sigma \, \lambda_{v+v'} + {\herm{u}{v'}_1}^\sigma \, \lambda_{v+v'}
\end{split} \]
and hence ${\herm{u}{v}_1}^\sigma (\lambda_{v+v'}-\lambda_v) = {\herm{u}{v'}_1}^\sigma (\lambda_{v'}-\lambda_{v+v'})$. Assume now that $\lambda_{v+v'} \neq \lambda_{v}$. Choose a $\tilde u \in H_1$ such that $\herm{\tilde u}{v'}_1 = 1$, and put $\beta = \herm{\tilde u}{v}_1$. Then $\beta^\sigma = (\lambda_{v'}-\lambda_{v+v'}) (\lambda_{v+v'}-\lambda_v)^{-1}$ and we have
\[ 0 \;=\; \herm{u}{v}_1 - \herm{u}{v'}_1 \, \beta
\;=\; \herm{u}{v - \beta^{\ast_1} v'}_1 \]
for all $u \in H_1$. Again using that $\herm{\cdot}{\cdot}_1$ is non-degenerate, we conclude $v = \beta^{\ast_1} v'$, that is, $v$ and $v'$ are linearly dependent, a contradiction. We conclude that $\lambda_{v+v'} = \lambda_{v}$ and similarly we see that $\lambda_{v+v'} = \lambda_{v'}$, that is, $\lambda_{v} = \lambda_{v'}$. As for any two vectors in $H_1$ there is a third vector linearly independent of both, we have shown that~(\ref{fml:quasiunitary}) holds for all $u, v \in H_1$. As regards the uniqueness assertion, we just note that there are $u, v \in H_1$ such that $\herm{u}{v}_1 = 1$.
\end{proof}

\begin{corollary} \label{thm:Piziak}
Let $H_1$ and $H_2$ be Hermitian spaces and assume that $H_1$ is at least $2$-dimensional. Let $\phi \colon H_1 \to H_2$ be a quasilinear bijection. Then $\phi$ is quasiunitary if and only if $\phi$ preserves the orthogonality relation.
\end{corollary}

\begin{proof}
The ``only if'' part is clear from Proposition~\ref{prop:quasiunitary-map-induces-orthoiso}. The ``if'' part holds by Theorem~\ref{thm:Piziak-original}.
\end{proof}

We are led to a version of Wigner's Theorem for Hermitian spaces.

\begin{theorem} \label{thm:Wigner-iso}
Let $H_1$ and $H_2$ be at least $3$-dimensional Hermitian spaces. Then any orthoisomorphism between $\P(H_1)$ and $\P(H_2)$ is induced by a quasiunitary map from $H_1$ to $H_2$.
\end{theorem}

\begin{proof}
Let $f \colon \P(H_1) \to \P(H_2)$ be an orthoisomorphism. Then, for any $u, v, w \in H_1$, we have $\lin u \in \{ \lin v, \lin w \}\cc$ if and only if $f(\lin u) \in \{ f(\lin v), f(\lin w) \}\cc$. By Lemma~\ref{lem:perp-determines-star}, this means that $\lin u \in \lin v \star \lin w$ if and only if $f(\lin u) \in f(\lin v) \star f(\lin w)$. By Corollary~\ref{cor:Fundamental-Theorem-Faure-Froelicher-2}, there is a quasilinear bijection $\phi \colon H_1 \to H_2$ such that $f = \P(\phi)$. As $\phi$ preserves the orthogonality relation, the assertion follows from Corollary~\ref{thm:Piziak}.
\end{proof}

We note that also quasiunitary maps inducing a certain orthoisomorphism are determined up to a factor only.

\begin{lemma} \label{lem:uniqueness-of-quasilinear-maps}
Let $\phi \colon H_1 \to H_2$ be a quasiunitary map between Hermitian spaces and assume that $H_1$ is at least $2$-dimensional. Let $\psi \colon H_1 \to H_2$ be a further semilinear map. Then $\P(\phi) = \P(\psi)$ if and only if there is a $\kappa \in F_2 \setminus \{0\}$ such that $\psi = \kappa \phi$. In this case, also $\psi$ is quasiunitary.
\end{lemma}

\begin{proof}
The asserted equivalence holds by Lemma~\ref{lem:uniqueness-of-semilinear-maps}. Furthermore, we readily check that for any $\kappa \in F_2 \setminus \{0\}$, also $\kappa \phi$ is quasiunitary.
\end{proof}

We may formulate in this context a statement similarly to Lemma~\ref{lem:quasilinear-to-linear-for-Hermitian-spaces}. Given a quasiunitary map $\phi \colon H_1 \to H_2$, we may make $H_2$ into a space over the scalar sfield of $H_1$, to get a unitary map from $H_1$ to the modified space $H_2'$. We note that in this case, by Remark~\ref{rem:quasiunitary},  also the sfield involutions associated with the Hermitian forms of $H_1$ and $H_2'$ coincide.

\begin{lemma} \label{lem:quasiunitary-to-unitary}
Let $H_1$ and $H_2$ be Hermitian spaces over the sfields $F_1$ and $F_2$, respectively. Let $\phi \colon H_1 \to H_2$ be a quasiunitary map and let $f = \P(\phi)$. Then there is a Hermitian space $H_2'$ over $F_1$ and an orthoisomorphism $t \colon \P(H_2) \to \P(H_2')$ such that $t \circ f$ is induced by a unitary map from $H_1$ to $H_2'$.
\end{lemma}

\begin{proof}
Let again $\sigma \colon F_1 \to F_2$ be the isomorphism associated with $\phi$ and let $\lambda \in F_2$ be such that~(\ref{fml:quasiunitary}) holds for $u, v \in H_1$. Let $H_2'$ be the linear space over $F_1$ that is constructed from $H_2$ as in the proof of Lemma~\ref{lem:quasilinear-to-linear}. This time, we define $\herm{\cdot}{\cdot}_2'$ by $\herm{u}{v}_2' = (\herm{u}{v}_2 \, \lambda^{-1})^{\sigma^{-1}}$ for $u, v \in H_2$.

Recall that, by Remark~\ref{rem:quasiunitary}, $\lambda = \lambda^{\ast_2}$. Hence $\ast_2'$ given by $\alpha^{\ast_2'} = \lambda^{\sigma^{-1}} \alpha^{\sigma \ast_2 \sigma^{-1}} (\lambda^{\sigma^{-1}})^{-1}$ is an involutory antiautomorphism on $F_1$. Moreover, a somewhat tedious but straightforward calculation shows that $\herm{\cdot}{\cdot}_2'$ is a Hermitian form on $H_2'$ w.r.t.\ $\ast_2'$.

Again, $\tau \colon H_2 \to H_2' \komma u \mapsto u$ is $\sigma^{-1}$-linear and hence $\tau \circ \phi$ is linear. Moreover, $\tau \circ \phi$ preserves the Hermitian form, that is, it is a unitary map. Clearly, $t = \P(\tau)$ is an orthoisomorphism. As $t \circ f = \P(\tau \circ \phi)$, the lemma follows.
\end{proof}

The following consequence of Theorem~\ref{thm:Wigner-iso} is also shown in \cite{PaVe3}. By an {\it ortho\-automorphism}, we mean an orthoisomorphism of an orthoset with itself.

\begin{corollary} \label{cor:Wigner-auto}
Let $H$ be an at least $3$-dimensional Hermitian space and let $f$ be an orthoautomorphism of $\P(H)$. Assume moreover that there is an at least $2$-dimen\-sional subspace $S$ of $H$ such that $f|_{\P(S)} = \id_{\P(S)}$. Then there is a unique unitary map $\phi \colon H \to H$ inducing $f$ such that $\phi|_S = \id_S$.
\end{corollary}

\begin{proof}
By Theorem~\ref{thm:Wigner-iso}, there is quasiunitary map $\psi \colon H \to H$ inducing $f$. Moreover, $\P(\psi)$ is the identity on $\P(S)$, hence by Lemma~\ref{lem:uniqueness-of-semilinear-maps} there is a non-zero scalar $\kappa$ such that $\psi|_S = \kappa \, \id_S$. Put $\phi = \kappa^{-1} \psi$. By Lemma~\ref{lem:uniqueness-of-quasilinear-maps}, also $\phi$ is quasiunitary. Moreover, $\phi|_S = \id_S$ is unitary and hence $\phi$ itself is unitary. The uniqueness of $\phi$ is again clear from Lemma~\ref{lem:uniqueness-of-semilinear-maps}.
\end{proof}

\section{Orthomodular spaces and linear Dacey spaces}
\label{sec:Orthomodular-spaces-Dacey-spaces}

The condition on a map to preserve the orthogonality relation is not only of interest if the map is bijective. Consider a Hermitian space $H$ and an orthoclosed subspace $S$ of $H$. The most elementary examples of maps between $S$ and $H$ are, first, the inclusion map $\iota \colon S \to H$ and, second, a map $\sigma \colon H \to S$ such that $\sigma|_S = \id_S$ and $\kernel \sigma = S\c$. The map $\iota$ is linear but not necessarily adjointable; a linear map $\sigma$ with the indicated properties might not exist. We deal in this section with Hermitian spaces that behave in this respect in the desired way.

\begin{definition}
An {\it orthomodular space} is a Hermitian space $H$ such that every orthoclosed subspace $S$ of $H$ is splitting.
\end{definition}

\begin{proposition} \label{prop:orthomodular-space}
For a Hermitian space $H$, the following are equivalent.
\begin{itemize}

\item[\rm (a)] $H$ is an orthomodular space.

\item[\rm (b)] For any orthoclosed subspace $S$ of $H$, there is a linear map $\sigma \colon H \to S$ such that $\sigma|_S = \id_S$ and $\kernel \sigma = S\c$.

\item[\rm (c)] For any orthoclosed subspace $S$ of $H$, the inclusion map $\iota \colon S \to H$ is adjointable.

\end{itemize}
In this case, let $S$ be an orthoclosed subspace of $H$. Then the map $\sigma \colon H \to S$ as specified in {\rm (b)} is the adjoint of the inclusion map $\iota \colon S \to H$.
\end{proposition}

\begin{proof}
This is easily checked.
\end{proof}

The orthosets relevant in this context are the following \cite{Dac,Wlc}.

\begin{definition}
A {\it Dacey space} is an orthoset $X$ such that the following holds: if $A$ and $B$ are subspaces of $X$ such that $B \subseteq A$, then $B$ is a subspace of $A$.
\end{definition}

We note that Dacey spaces can alternatively be defined lattice-theoretically. Given an orthoset $X$, we denote by ${\mathsf C}(X)$ the ortholattice of subspaces of $X$; cf., e.g., \cite[Section 2]{PaVe4}.

\begin{lemma} \label{lem:Dacey-orthomodularity}
An orthoset $X$ is a Dacey space if and only if ${\mathsf C}(X)$ is orthomodular.
\end{lemma}

\begin{proof}
See \cite[Lemma 2.15]{PaVe4}.
\end{proof}

For linear orthosets, the property of being a Dacey space can be expressed as follows.

\begin{lemma} \label{lem:linear-Dacey-space}
Let $X$ be a linear orthoset. Then $X$ is a Dacey space if and only if the following holds: for any subspace $A$ of $X$ and any $x \in X$, there are $y \in A$ and $z \in A\c$ such that $x \in \{y, z\}\cc$.
\end{lemma}

\begin{proof}
By \cite[Lemma 4.4 and Corollary 3.11]{PaVe3}, ${\mathsf C}(X)$ is an atomistic ortholattice with the covering property. By \cite[Lemma 5.3]{PaVe1}, the atoms of ${\mathsf C}(X)$ are $\{x\}\cc = \{0,x\}$, $x \in X \setminus \{0\}$.

By Lemma~\ref{lem:Dacey-orthomodularity}, $X$ is Dacey if and only if ${\mathsf C}(X)$ is orthomodular. Hence the assertion holds by \cite[Lemma~(30.7)]{MaMa}.
\end{proof}

We arrive at the following well-known characterisation of orthomodular spaces, see, e.g., \cite[Theorem (2.8)]{Piz2}.

\begin{proposition} \label{prop:orthomodular-space-Dacey-space}
A Hermitian space $H$ is orthomodular if and only if $\P(H)$ is a Dacey space.
\end{proposition}

\begin{proof}
For $H$ to be orthomodular means that for any orthoclosed subspace $S$ of $H$ and any $u \in H$ there are $v \in S$ and $w \in S\c$ such that $u = v + w$. This is case iff, for any orthoclosed subspace $\P(S)$ of $\P(H)$ and $\lin u \in \P(H)$, there are $\lin v \in \P(S)$ and $\lin w \in \P(S)\c = \P(S\c)$ such that $\lin u \in \{ \lin v, \lin w \}\cc$. Hence the assertion follows from Lemma~\ref{lem:linear-Dacey-space}.
\end{proof}

\begin{theorem} \label{thm:orthomodular-spaces-Dacey-spaces}
Let $H$ be an orthomodular space. Then $\P(H)$ is a linear Dacey space.

Conversely, let $X$ be a linear Dacey space of rank $\geq 4$. There there is an orthomodular space $H$ and an orthoisomorphism between $X$ and $\P(H)$.
\end{theorem}

\begin{proof}
This is clear from Theorem~\ref{thm:Hermitian-spaces-linear-orthosets} and Proposition~\ref{prop:orthomodular-space-Dacey-space}.
\end{proof}

For dimensions $\geq 4$, orthomodular spaces may thus be described as linear Dacey spaces. There is a further way of specifying the orthosets corresponding to orthomodular spaces, namely, by means of the adjointability of the inclusion maps. This was the idea that led to the representation theorem of orthomodular spaces in \cite{LiVe} and the following theorem is an alternative formulation of \cite[Theorem~1.8]{LiVe}.

We call an orthoset $X$ {\it irreducible} if $X \setminus \{0\}$ cannot be partitioned into two non-empty orthogonal subsets of $X$. Moreover, we call $X$ {\it Fr\' echet} if, for any proper elements $x, y \in X$, $\{x\}\c \subseteq \{y\}\c$ implies $x = y$.

\begin{theorem} \label{thm:orthomodular-space-as-orthoset}
Let $X$ be an irreducible Fr\' echet orthoset of rank $\geq 4$. Assume that for any subspace $A$ of $X$, the inclusion map $\iota \colon A \to X$ is adjointable. Then there is an orthomodular space $H$ and an orthoisomorphism between $X$ and $\P(H)$.
\end{theorem}

\begin{proof}
As $X$ is Fr\' echet, it follows by \cite[Lemma 2.10]{PaVe4} that ${\mathsf C}(X)$ is atomistic. By \cite[Proposition 4.10]{PaVe4}, ${\mathsf C}(X)$ is an orthomodular lattice with the covering property. By \cite[Lemma 2.14]{PaVe4}, ${\mathsf C}(X)$ is irreducible. Hence, by \cite[Theorem (34.5)]{MaMa}, there is a Hermitian space $H$ such that $X$ and $\P(H)$ are orthoisomorphic. By Proposition~\ref{prop:orthomodular-space-Dacey-space}, $H$ is orthomodular.
\end{proof}

Adding a transitivity condition, we may moreover characterise the infinite-dimension\-al classical Hilbert spaces; cf.~\cite[Theorem 5.3]{LiVe}.

\begin{theorem} \label{thm:classical-HS-as-lattice}
Let $X$ be an irreducible Fr\' echet orthoset of infinite rank. Assume that {\rm (i)}~for any subspace $A$ of $X$, the inclusion map $\iota \colon A \to X$ is adjointable and {\rm (ii)}~for any proper elements $x$ and $y$ of $X$, there is an orthoautomorphism $f$ of $X$ such that $f(x) = y$ and $f(z) = z$ \/ for any $z \perp x, y$. Then there is a Hilbert space $H$ over $\Reals$, $\Complexes$, or $\Quaternions$ and an orthoisomorphism between $X$ and $\P(H)$.
\end{theorem}

\begin{proof}
By Theorem~\ref{thm:orthomodular-space-as-orthoset}, there is an orthomodular space $H$ such that $\P(H)$ and $X$ are orthoisomorphic. Let $u, v \in H \setminus \{0\}$. By assumption, there is an orthoautomorphism $f$ of $\P(H)$ such that $f(\lin u) = \lin v$ and $f(\lin w) = \lin w$ for any $w \perp u, v$. By Corollary~\ref{cor:Wigner-auto}, $f$ is induced by a unitary map $\phi$ on $H$. We conclude that there is a non-zero scalar $\alpha$ such that each one-dimensional subspace of $H$ contains a vector of length $\alpha$. The assertion follows now from Sol\` er's Theorem \cite{Sol}.
\end{proof}

\section{Partial orthometric correspondences}
\label{sec:Partial-orthometries}

The aim of our final section is to generalise the results of Section~\ref{sec:Orthometric-correspondences}: we want to consider maps between Hermitian spaces that do not necessarily establish an orthometric correspondence between the whole spaces but only between certain subspaces. General Hermitian spaces are not eligible to this concern, we hence restrict our considerations to orthomodular spaces.

For a map $f \colon A \to B$ and sets $A_0 \subseteq A$ and $\image f \subseteq B_0 \subseteq B$, we denote by $f|_{A_0}^{B_0}$ the map $f$ restricted to $A_0$ and corestricted to $B_0$.

\begin{definition} \label{def:partial-quasiisometry}
Let $H_1$ be an orthomodular space over the sfield $F_1$ and $H_2$ an orthomodular space over the sfield $F_2$. Let $S_1$ be an orthoclosed subspace of $H_1$ and $S_2$ an orthoclosed subspace of $H_2$. Assume that $\phi \colon H_1 \to H_2$ is a quasilinear map such that $\kernel \phi = {S_1}\c$, $\image \phi = S_2$, and $\phi|_{S_1}^{S_2}$ is quasiunitary. Then we call $\phi$ a {\it partial quasiisometry}. If $\phi$ is linear and $\phi|_{S_1}^{S_2}$ is unitary, we call $\phi$ a {\it partial isometry}.
\end{definition}

Let $\phi \colon H_1 \to H_2$ be a partial isometry. Thanks to Proposition \ref{prop:orthomodular-space}, we may define its generalised inverse $\psi$ in the expected way. Let $S_1 = (\kernel \phi)\c$ and $S_2 = \image \phi$; then we require that $\kernel \psi = {S_2}\c$ and $\psi|_{S_2}^{S_1} = (\phi|_{S_1}^{S_2})^{-1}$.

\begin{proposition}
Let $\phi \colon H_1 \to H_2$ be a partial isometry between orthomodular spaces. Then $\phi$ is adjointable and $\phi\adj$ is the generalised inverse of $\phi$.
\end{proposition}

\begin{proof}
Let $S_1 = (\kernel \phi)\c$ and $S_2 = \image \phi$. Let $\iota_1 \colon S_1 \to H_1$ be the inclusion map and $\sigma_1 \colon H_1 \to S_1$ the projection map, and define $\iota_2$ and $\sigma_2$ similarly. By Propositions~\ref{prop:unitary-map} and~\ref{prop:orthomodular-space}, $\sigma_1$, $\phi|_{S_1}^{S_2}$, and $\iota_2$ are adjointable and we have $\phi\adj = (\iota_2 \circ \phi|_{S_1}^{S_2} \circ \sigma_1)\adj = \sigma_1\adj \circ (\phi|_{S_1}^{S_2})\adj \circ \iota_2\adj = \iota_1 \circ (\phi|_{S_1}^{S_2})^{-1} \circ \sigma_2$. This map is the generalised inverse of $\phi$.
\end{proof}

\begin{definition} \label{def:partial-orthometry}
Let $f \colon X \to Y$ be an adjointable map between Dacey spaces. We call $f$ a {\it partial orthometry} if there are subspaces $A$ of $X$ and $B$ of $Y$ such that $f$ establishes an orthoisomorphism between $A$ and $B$ and $f(x) = 0$ for $x \perp A$.

In this case, let $g \colon Y \to X$ be an adjointable map with the following property: $g$ establishes an orthoisomorphism between $B$ and $A$ such that $g|_B^A = (f|_A^B)^{-1}$, and $g(y) = 0$ for $y \perp B$. Then $g$ is called a {\it generalised inverse} of $f$.
\end{definition}

In case when $g$ is a generalised inverse of a partial orthometry $f$, it is obvious that then $g$ is likewise a partial orthometry and $f$ is a generalised inverse of $g$.

For an adjointable map $f \colon X \to Y$ between orthosets, we put $\zerokernel f = f\big|_{(\kernel f)\c}^{(\image f)\cc}$, called the {\it zero-kernel restriction} of $f$ \cite[Section 3]{PaVe4}.

Our next Proposition describes partial orthometries between irredundant Dacey spaces. We recall once more that, by Lemma~\ref{lem:unique-adjoint}, adjoints between irredundant orthosets are unique.

\begin{proposition} \label{prop:partial-orthometry-between-Dacey-spaces}
Let $f \colon X \to Y$ be a partial orthometry between irredundant Dacey spaces. Then the adjoint $f\adj$ of $f$ is the unique generalised inverse of $f$.

Moreover, let $A$ and $B$ be the subspaces of $X$ and $Y$ such that, in accordance with Definition~\ref{def:partial-orthometry}, $f|_A^B$ is an orthoisomorphism. Then $A = (\kernel f)\c = \image f\adj$ and $B = (\kernel f\adj)\c = \image f$. Moreover, the inclusion maps $\iota_A \colon A \to X$ and $\iota_B \colon B \to Y$ are adjointable; let $\sigma_A$ be the adjoint of $\iota_A$ and $\sigma_B$ the adjoint of $\iota_B$. We have $\zerokernel f = f|_A^B$ and $\zerokernel{f\adj} = f\adj|_B^A = (f|_A^B)^{-1}$, and
\begin{equation} \label{fml:partial-orthometry-between-Dacey-spaces}
f \;=\; \iota_B \circ \zerokernel f \circ \sigma_A,
\quad f\adj \;=\; \iota_A \circ \zerokernel{f\adj} \circ \sigma_B
\;=\; \iota_A \circ (\zerokernel f)^{-1} \circ \sigma_B.
\end{equation}
\end{proposition}

\begin{proof}
We have $A\c \subseteq \kernel f$ and $A \cap \kernel f = \{0\}$. As $X$ is a Dacey space, $(\kernel f)\c \subseteq A$ implies that $(\kernel f)\c$ is a subspace of $A$ and hence $(\kernel f)\c = ((\kernel f)\cc \cap A)\c \cap A = (\kernel f \cap A)\c \cap A = A$. Moreover, $\image f\adj \subseteq (\kernel f)\c = A$ and $\kernel f\adj = (\image f)\c \subseteq B\c$ by \cite[Lemma 3.8]{PaVe4}. As $y \in B\c$ means $y \perp f(A)$ and hence $f\adj(y) \perp A$, we have $y \in \kernel f\adj$, and we conclude $\kernel f\adj = B\c$, that is, $B = (\kernel f\adj)\c$. Also $B \subseteq \image f \subseteq (\kernel f\adj)\c = B$, that is, $\image f = B$.

It is clear now that $\zerokernel f = f|_A^B$. As $(\image f\adj)\cc = (\kernel f)\c = A$, we also have $\zerokernel{f\adj} = f\adj|_B^A$. The map $\zerokernel f$ is an orthoisomorphism whose adjoint is $\zerokernel{f\adj}$, hence by Proposition~\ref{prop:orthoisomorphism} $\zerokernel{f\adj} = (\zerokernel f)^{-1}$. In particular, $\image f\adj = A$, and $f\adj$ is a generalised inverse of $f$.

We further conclude that $\sigma_A = (f\adj \circ f)|^A$. Indeed, $(f\adj \circ f)|^A$ is the adjoint of $\iota_A$ because, for $x \in X$ and $y \in A$, we have $f\adj(f(x)) \perp y$ iff $x \perp f\adj(f(y)) = y = \iota_A(y)$. We calculate for $x \in X$
\[ \iota_B(\zerokernel f(\sigma_A(x)))
\;=\; f(\sigma_A(x))
\;=\; f(f\adj(f(x)))
\;=\; f(x). \]
Hence the first equation in~(\ref{fml:partial-orthometry-between-Dacey-spaces}) holds and the second is seen similarly.

The first equation in~(\ref{fml:partial-orthometry-between-Dacey-spaces}) shows that $f$ is uniquely determined by its zero-kernel restriction $\zerokernel f = f|_A^B$. It follows that the generalised inverse is unique as well.
\end{proof}

\begin{proposition} \label{prop:quasiisometry-induces-parial-orthometry}
Let $\phi \colon H_1 \to H_2$ be a partial quasiisometry between orthomodular spaces. Then $\P(\phi)$ is a partial orthometry.
\end{proposition}

\begin{proof}
Let $S_1$ and $S_2$ be the orthoclosed subspaces of $H_1$ and $H_2$ such that, in accordance with Definition~\ref{def:partial-quasiisometry}, $\phi|_{S_1}^{S_2}$ is quasiunitary. Let $\sigma_{S_1}$ and $\sigma_{S_2}$ be the projection maps onto $S_1$ and $S_2$. Define $\psi \colon H_2 \to H_1 \komma v \mapsto (\phi|_{S_1}^{S_2})^{-1}(\sigma_{S_2}(v))$. For any $u \in H_1$ and $v \in H_2$, we then have that $\phi(u) \perp v$ if and only if $\phi(\sigma_{S_1}(u)) \perp \sigma_{S_2}(v)$ if and only if $\sigma_{S_1}(u) \perp (\phi|_{S_1}^{S_2})^{-1}(\sigma_{S_2}(v))$ if and only if $u \perp \psi(v)$. Hence $\P(\psi)$ is an adjoint of $\P(\phi)$.

In particular, $\P(\phi)$ is adjointable. Now the assertion is clear by Proposition~\ref{prop:quasiunitary-map-induces-orthoiso}.
\end{proof}

\begin{theorem} \label{thm:Wigner-iso-for-partial-orthometries}
Let $H_1$ and $H_2$ be orthomodular spaces and let $f \colon \P(H_1) \to \P(H_2)$ be a partial orthometry such that $(\kernel f)\c$ has rank $\geq 3$. Then $f$ is induced by a partial quasiisometry from $H_1$ to $H_2$.
\end{theorem}

\begin{proof}
Let $S_1$ and $S_2$ be the orthoclosed subspaces of $H_1$ and $H_2$, respectively, such that $\kernel f = \P(S_1)\c$, $\image f = \P(S_2)$, and $\zerokernel f = f|_{\P(S_1)}^{\P(S_2)}$ is an orthoisomorphism. By Proposition~\ref{prop:partial-orthometry-between-Dacey-spaces}, $f = \iota_{\P(S_2)} \circ \zerokernel f \circ \sigma_{\P(S_1)}$, where $\iota_{\P(S_2)} \colon \P(S_2) \to \P(H_2)$ is the inclusion map and $\sigma_{\P(S_1)} \colon \P(H_1) \to \P(S_1)$ is the adjoint of the inclusion map $\iota_{\P(S_1)} \colon \P(S_1) \to \P(H_1)$. By Theorem~\ref{thm:Wigner-iso}, there is a quasiunitary map $\hat \phi \colon S_1 \to S_2$ such that $\zerokernel f = \P(\hat \phi)$. Let $\iota_{S_1} \colon S_1 \to H_1$ and $\iota_{S_2} \colon S_2 \to H_2$ be the inclusion maps, with adjoints $\sigma_{S_1}$ and $\sigma_{S_2}$, respectively. We clearly have $\iota_{\P(S_1)} = \P(\iota_{S_1})$ and $\iota_{\P(S_2)} = \P(\iota_{S_2})$. Moreover, by Proposition~\ref{prop:adjointable-linear-map}, $\sigma_{\P(S_1)} = \P(\sigma_{S_1})$. We conclude that $f$ is induced by the partial quasiisometry $\iota_{S_2} \circ \hat \phi \circ \sigma_{S_1}$.
\end{proof}

One more time, we may make a partial quasiisometry into a partial isometry.

\begin{lemma} \label{lem:quasiisometry-to-isometry}
Let $H_1$ and $H_2$ be Hermitian spaces over the sfields $F_1$ and $F_2$, respectively. Let $\phi \colon H_1 \to H_2$ be a partial quasiisometry and let $f = \P(\phi)$. Then there is a Hermitian space $H_2'$ over $F_1$ and an orthoisomorphism $t \colon \P(H_2) \to \P(H_2')$ such that $t \circ f$ is induced by a partial isometry from $H_1$ to $H_2'$.
\end{lemma}

\begin{proof}
Let $S_1 = (\kernel \phi)\c$ and $S_2 = \image \phi$. Then $S_1$ and $S_2$ are orthoclosed subspaces and $\phi|_{S_1}^{S_2}$ is quasiunitary. Let $H_2'$ be the Hermitian space over $F_1$ constructed from $H_2$ as in the proof of Lemma~\ref{lem:quasiunitary-to-unitary} and let $\tau \colon H_2 \to H_2' \komma u \mapsto u$. Put $S_2' = \tau(S_2)$. Then $\tau \circ \phi$ is linear and $(\tau \circ \phi)|_{S_1}^{S_2'}$ is unitary. Moreover, $S_1 = \kernel (\tau \circ \phi)\c$ and $S_2' = \image (\tau \circ \phi)$ are orthoclosed subspaces. We conclude that $\tau \circ \phi$ is a partial isometry.
\end{proof}

\begin{lemma} \label{lem:Fundamental-Theorem-for-isometry}
Assume that $S$ is an at least $3$-dimensional subspace of the orthomodular space $H$. Let $f \colon \P(H) \to \P(H)$ be a partial orthometry such that $f|_{\P(S)} = \id_{\P(S)}$. Then $f$ is induced by a unique partial isometry $\phi \colon H_1 \to H_2$ such that $\phi|_S = \id_S$.
\end{lemma}

\begin{proof}
By Theorem~\ref{thm:Wigner-iso-for-partial-orthometries}, there is a partial quasiisometry $\psi \colon H \to H$ such that $f = \P(\psi)$. As in the proof of Corollary~\ref{cor:Wigner-auto}, we see that there is a $\kappa \neq 0$ such that $\phi = \kappa^{-1} \psi$ is on $S$ the identity. Then $\phi|_S$ is unitary and we conclude that $\phi$ is a partial isometry.
\end{proof}

\subsubsection*{Acknowledgements}

This research was funded in part by the Austrian Science Fund (FWF) 10.55776/ \linebreak PIN5424624 and the Czech Science Foundation (GACR) 25-20013L.

\end{document}